\def\timestamp{%
Time-stamp: <P-domination-two.tex: Tuesday 12-01-2016 at 13:02:50 (cet)>}
\def\stripname Time-stamp: <#1 #2>{#2}
\edef\filedate{\expandafter\stripname\timestamp}

\documentclass{amsart}

\DeclareMathSymbol\restr\mathbin{AMSa}{"16}  
\DeclareMathSymbol\PP\mathord{AMSb}{`P}
\DeclareMathSymbol\le\mathrel{AMSa}{"36}
\DeclareMathSymbol\ge\mathrel{AMSa}{"3E} 
\DeclareMathSymbol\emptyset\mathord{AMSb}{"3F}

\newcommand\bee{\mathfrak{b}}
\newcommand\cee{\mathfrak{c}}
\newcommand\dee{\mathfrak{d}}
\newcommand\Fn{\operatorname{Fn}}
\newcommand\dom{\operatorname{dom}}

\newcommand\calK{\mathcal{K}}
\newcommand\MA{\mathsf{MA}}

\newcommand\ourprod{2^{\omega_1}}
\newcommand\ourtree{2^{<\omega_1}}
\newcommand\preim{^\gets}
\newcommand\emptyseq{\langle\,\rangle}
\newcommand\omegaseq[2][n]{\langle{#2}_#1:#1\in\omega\rangle}
\newcommand\deltaomegaone{{\omega_1\setminus\delta}}
\newcommand\omegaomegaone{{\omega_1\setminus\omega}}

\theoremstyle{plain}
\newtheorem{theorem}{Theorem}%[section]

\newtheorem{lemma}[theorem]{Lemma}

\theoremstyle{remark}
\newtheorem{remark}[theorem]{Remark}

\newcommand\orpr[2]{\langle{#1},{#2}\rangle}

\usepackage{amsrefs}

\begin{document}

\title[Compact spaces with a $\PP$-diagonal]%
      {Compact spaces with a $\PP$-diagonal}

\author{Alan Dow}
\address{
Department of Mathematics\\
UNC-Charlotte\\
9201 University City Blvd. \\
Charlotte, NC 28223-0001}
\email{adow@uncc.edu}
\urladdr{http://www.math.uncc.edu/\~{}adow}

\author{Klaas Pieter Hart}
\address{Faculty of Electrical Engineering, Mathematics and Computer Science\\
         TU Delft\\
         Postbus 5031\\
         2600~GA {} Delft\\
         the Netherlands}
\email{k.p.hart@tudelft.nl}
\urladdr{http://fa.its.tudelft.nl/\~{}hart}



\begin{abstract}
  We prove that compact Hausdorff spaces with a $\PP$-diagonal
  are metrizable.
  This answers problem~4.1 (and the equivalent problem~4.12)
  from~\cite{MR2739891}. 
\end{abstract}

\date{\filedate}
\maketitle

\section*{Introduction}

The purpose of this note is to show that a compact space with a $\PP$-diagonal
is metrizable.

To explain the meaning of this statement we need to introduce a bit of notation
and define a few notions.
For a space~$M$ (always assumed to be at least completely regular) we
let $\calK(M)$ denote the family of compact subsets of~$M$.
Following~\cite{MR2150789} we say that a space~$X$ is \emph{$M$-dominated}
if there is a cover $\{C_K:K\in\calK(M)\}$ of~$X$ by compact subsets
with the property that $K\subseteq L$ implies $C_K\subseteq C_L$.

In the case that we deal with, namely where $M$~is the space of irrational
numbers, we can simplify the cover a bit and make it more amenable
to combinatorial treatment.
The space of irrationals is homeomorphic to the product space $\omega^\omega$,
where $\omega$~carries the discrete topology.
We shall reserve the letter~$\PP$ for this space.

The set~$\PP$ is ordered coordinatewise: $f\le g$ means
$(\forall n)(f(n)\le g(n))$.
Using this order we simplify the formulation of $\PP$-dominated as follows.
If $K$~is a compact subset of~$\PP$ then the function~$f_K$, given
by $f_K(n)=\max\{g(n):g\in K\}$, is well-defined.
Using this one can easily verify that a space~$X$ is $\PP$-dominated
iff there is a cover $\langle K_f:f\in\PP\rangle$ of~$X$ by compact sets
such that $f\le g$ implies $K_f\subseteq K_g$.
We shall call such a cover \emph{an order-preserving cover by compact sets}.

Finally then we say that a space~$X$ has a $\PP$-diagonal if the complement
of the diagonal, $\Delta$, in~$X^2$ is $\PP$-dominated.
Problem~4.1 from~\cite{MR2739891} asks whether a compact space with
a $\PP$-diagonal is metrizable.
The authors of that paper proved that the answer is positive if $X$~is
assumed to have countable tightness, or in general if $\MA(\aleph_1)$
is assumed.
The latter proof used that assumption to show that $X$~has a small diagonal,
which in turn implies that $X$~has countable tightness so that the first
result applies.
Thus, Problem~4.12 from~\cite{MR2739891}, which asks if a compact
space with a $\PP$-diagonal has a small diagonal, is a natural reformulation
of Problem~4.1.

The property of $\PP$-domination arose in the study of the geometry of
topological vector space; in \cite{MR895307} it was shown that if a locally
convex space has a form of $\PP$-domination then its compact sets
are metrizable.
The paper~\cite{MR2739891} contains more information and results
leading up to its Problem~4.1.

The main result of~\cite{MR3338973} states that compact spaces with
a $\PP$-diagonal are metrizable under the assumption of the Continuum
Hypothesis.
The proof establishes that a compact space with a $\PP$-diagonal that
has \emph{uncountable} tightness maps onto the Tyhconoff
cube~$[0,1]^{\omega_1}$ and no compact space with a $\PP$-diagonal
maps onto the cube~$[0,1]^\cee$.

The principal result of this paper closes the gap between $\aleph_1$
and $\cee$ by establishing that no compact space with a $\PP$-diagonal
maps onto~$[0,1]^{\omega_1}$.

Furthermore we would like to point out that Lemma~\ref{lemma.KfBIG}
establishes a Baire category type property of~$\ourprod$:
in an order-preserving cover by compact sets there
are many members with non-empty interior in the $G_\delta$-topology.

\section*{Some preliminaries}

In the proof of the main lemma, Lemma~\ref{lemma.KfBIG},
we need to consider three cases, depending on the values of the
familiar cardinals~$\bee$ and~$\dee$.
These are defined in terms of the mod~finite order on~$\PP$:
we say $f\le^* g$ if $\{n:g(n)<f(n)\}$ is finite.
Then $\bee$~is the minimum size of a subset of~$\PP$ that is unbounded
with respect to~$\le^*$, and $\dee$~is the minimum size of a dominating
(i.e., cofinal) set with respect to~$\le^*$.
Interestingly, $\dee$~is also the minimum size of a dominating set with respect
to the coordinatewise order~$\le$; we shall use this in the proof of
the main lemma.
We refer to Van~Douwen's~\cite{MR776622} for more information.

Since we shall be working with the Cantor cube $\ourprod$ we fix
a bit of notation.
If $I$ is some subset of~$\omega_1$ then $\Fn(I,2)$ denotes the set
of finite partial functions from~$I$ to~$2$.
We let $\ourtree$ denote the binary tree of countable sequences of zeros
and ones. 
If $s\in\Fn(\omega_1,2)$ then $[s]$ denotes $\{x\in\ourprod:s\subseteq x\}$;
the family $\{[s]:s\in \Fn(\omega_1,2)\}$ is the standard base for the
product topology of~$\ourprod$.
Similarly, if $\rho\in\ourtree$
then $[\rho]=\{x\in\ourprod:\rho\subseteq x\}$, and the family
$\{[\rho]:\rho\in\ourtree\}$ is the standard base for what is called
the $G_\delta$-topology on~$\ourprod$; a set dense with respect to this
topology will be called $G_\delta$-dense.

When working with powers of the form~$I^{\omega_1}$, where $I=\omega$ or $I=2$,
we use $\pi_\delta$ to denote the projection of~$I^{\omega_1}$
onto~$I^\deltaomegaone$.

In the proof of Lemma~\ref{lemma.KfBIG} we shall need the following
result, due to Todor\v{c}evi\'c.

\begin{lemma}[\cite{MR980949}*{Theorem~1.3}]  \label{lemma.stevo}
  If\/ $\bee=\aleph_1$ then $\omega^{\omega_1}$ has a subset, $X$, of
  cardinality~$\aleph_1$ such that for every $A\in[X]^{\aleph_1}$
  there are $D\in[A]^{\aleph_0}$ and $\delta\in\omega_1$ such that
  $\pi_\delta[D]=\{d\restr(\deltaomegaone):d\in D\}$ is dense
  in~$\omega^\deltaomegaone$.\qed  
\end{lemma}

Theorem~1.3 of \cite{MR980949} is actually formulated as a theorem
about~$\bee$: drop the assumption $\bee=\aleph_1$ and replace every~$\omega_1$
and~$\aleph_1$ by~$\bee$.
As explained in~\cite{MR980949} this shows that there are natural versions
of the S-space problem that do have ZFC solutions.

The lemma also holds with $\omega$ replaced by~$2$, simply map
$\omega^{\omega_1}$ onto~$\ourprod$ by taking all coordinates modulo~$2$.
In that case the density of~$\pi_\delta[D]$ can be expressed
by saying that for every $s\in\Fn(\deltaomegaone,2)$
the intersection~$D\cap[s]$ is nonempty.

\section*{BIG sets in $\ourprod$}

Let us call a subset, $Y$, of $\ourprod$ BIG if it is compact
and projects onto some final product, that is, there is a
$\delta\in\omega_1$ such that $\pi_\delta[Y]=2^\deltaomegaone$.
The latter condition can be expressed without mentioning projections as
follows: there is a $\delta\in\omega_1$ such that for every
$s\in\Fn(\deltaomegaone,2)$ the intersection $Y\cap[s]$ is
nonempty (and a dense set that is closed is equal to the whole space).

BIG sets are also big combinatorially, in the following sense.

\begin{lemma}\label{lemma.BIGisbig}
  If $Y$ is a BIG subset of $\ourprod$ then there is a node~$\rho$
  in the tree $\ourtree$ such that $[\rho]\subseteq Y$.
\end{lemma}

\begin{proof}
  Let $Y$ be BIG and fix a $\delta$ witnessing this.
  After reindexing we can assume $\delta=\omega$ and we
  let $B_t=\{x\in\ourprod:t\subset x\}$ and $Y_t=Y\cap B_t$ for $t\in2^{<\omega}$.

  Starting from $t_0=\emptyseq$ and $s_0=\emptyset$ we build
  a sequence $\omegaseq{t}$ in~$2^{<\omega}$
  and a sequence $\omegaseq{s}$
  in~$\Fn(\omegaomegaone,\omega)$ such that
  $[s_n]\subseteq\pi_\delta[Y_{t_n}]$ for all~$n$.

  Given $t_n$ we can choose $i_n<2$, and set $t_{n+1}=t_n*i_n$,
  such that $[s_n]\cap \pi_\delta[Y_{t_{n+1}}]$ has nonempty interior.
  Then choose an extension~$s_{n+1}$ of~$s_n$ such that
  $[s_{n+1}]\subseteq \pi_\delta[Y_{t_{n+1}}]$.
  With a bit of bookkeeping one can ensure that $\bigcup_n\dom s_n$
  is an initial segment of~$\omegaomegaone$.
  We let $\rho$ be the concatenation of~$\bigcup_nt_n$ and~$\bigcup_ns_n$.

  To see that $\rho$ is as required let $x\in[\rho]$.
  By construction we have $x\in[s_n]$ for all~$n$, so that, again for all~$n$,
  there is $y_n\in Y_{t_n}$ such that $y_n$ and $x$ agree above $\dom\rho$.
  If $s\in\Fn(\omega_1,2)$ determines a basic neighbourhood of~$x$ then
  there is an~$m$ such that $\dom s\cap\dom\rho$ is a subset
  of~$\dom t_m\cup\dom s_m$.
  Then $y_n\in[s]$ for all~$n\ge m$, so that the sequence $\omegaseq{y}$
  converges to~$x$, which shows that $x\in Y$.
\end{proof}

\section*{Existence of BIG sets}

It is clear that a compact space is $\PP$-dominated:
simply let $K_f$ be the whole space for all~$f$.
However, in our proof we shall encounter $\PP$-dominating covers that may
consist of proper subsets.
Our next result shows that such a cover of~$\ourprod$ by compact sets
must contain a BIG subset. 

\begin{lemma}%[$\bee=\aleph_1$]
  \label{lemma.KfBIG}
  If $\langle K_f:f\in\PP\rangle$ is an order-preserving cover
  of~$\ourprod$ by compact sets then there is an~$f$ such that
  $K_f$~is BIG.  
\end{lemma}

\begin{proof}
  We consider three cases.

  First we assume $\dee=\aleph_1$.
  In this case we show outright that there are $\rho\in\ourtree$
  and $f\in\PP$ such that $[\rho]\subseteq K_f$.
  Let $\langle f_\alpha:\alpha\in\omega_1\rangle$ be a sequence
  that is $\le$-dominating.

  Working toward a contradiction we assume no~$\rho$ and~$f$, as desired,
  can be found.
  This implies that for every~$\rho$ and every~$f$ the intersection
  $K_f\cap[\rho]$ is nowhere dense in~$[\rho]$.
  Indeed, if such an intersection has interior then there is
  $s\in\Fn(\omega_1,2)$ such that $[s]\cap[\rho]$ is nonempty
  and contained in~$K_f$.
  It would then be an easy matter to find $\sigma\in\ourtree$
  that extends both~$\rho$ and~$s$, and then $[\sigma]\subseteq K_f$.

  This allows us to choose an increasing sequence
  $\langle\rho_\alpha:\alpha\in\omega_1\rangle$ in~$\ourtree$
  such that $[\rho_\alpha]\cap K_{f_\alpha}=\emptyset$ for all~$\alpha$.
  Then the point $x=\bigcup_\alpha\rho_\alpha$ does not belong to any~$K_f$
  because the $K_{f_\alpha}$ are cofinal in the whole family.

  \smallskip
  Next we assume $\dee>\bee=\aleph_1$.
  We apply $\bee=\aleph_1$ to find a special subset~$X$
  of~$\ourprod$ as in the comment after Lemma~\ref{lemma.stevo}.
  In what follows, when $t\in\omega^{<\omega}$ we let $K(t)$ denote
  the union $\bigcup\{K_f:t\subseteq f\}$.

  We choose an increasing sequence $\omegaseq{t}$
  in $\omega^{<\omega}$, together with, for each~$n$,
  an uncountable subset~$A_n$ of~$X$, a countable subset~$D_n$ of~$A_n$,
  and $\delta_n\in\omega_1$ such that
  $A_n\subseteq K(t_n)$ and
  for all $s\in\Fn({\omega_1\setminus\delta_n},2)$
  the intersection $D_n\cap[s]$ is nonempty.
  Simply use that $K(t)=\bigcup_k K(t*k)$ for all~$t$.

  Let $\delta=\sup_n\delta_n$ and enumerate each $D_n$
  as $\langle d(n,m):m\in\omega\rangle$.
  
  For each $s\in\Fn(\deltaomegaone,2)$ each $D_n$ intersects~$[s]$
  so that we can define $h_s\in\omega^\omega$ by
  $h_s(n)=\min\{m:d(n,m)\in[s]\}$.

  By $\dee>\aleph_1$ there is $g\in\omega^\omega$ such that
  $\{n:h_s(n)<g(n)\}$ is infinite for all~$s$.
  
  Now let $E=\{d(n,m):m<g(n), n\in\omega\}$ and observe that
  $E$~meets~$[s]$ for every $s\in\Fn(\deltaomegaone,2)$,
  so that $\pi_\delta[E]$ is dense in $2^\deltaomegaone$.

  For each $n$ there is $f_n\in\PP$ that extends~$t_n$ and is such that
  $\{d(n,m):m<g(n)\}$~is a subset of~$K_{f_n}$.
  As $f_m(n)=t_{n+1}(n)$ if $m>n$ we may define $f\in\PP$ by
  $f(n)=\max\{f_m(n):m\in\omega\}$ for all~$n$.
  Thus we find a single~$f$ such that $E\subseteq K_f$, which
  immediately implies that $K_f$~is BIG.

  \smallskip
  Our last case is when $\bee>\aleph_1$.
  We let $A$ be the set of members, $t$, of~$\omega^{<\omega}$ for which there
  is a $\rho\in\ourtree$ such that $K(t)\cap[\rho]$ is $G_\delta$-dense
  in~$[\rho]$.

  As $K(\emptyseq)=\ourprod$ we have $\emptyseq\in A$.

  We show that if $t\in A$, as witnessed by~$\rho$,
  then there is an~$m_t$ such that $t*n\in A$
  whenever $n\ge m_t$; as $K(t*m)\subseteq K(t*n)$ whenever $m\le n$
  it follows that we need to find just one~$n$ such that $t*n\in A$.
  Build, recursively, an increasing
  sequence $\rho=\rho_0\subseteq\rho_1\subseteq\rho_2\subseteq\cdots$
  in~$\ourtree$ such that $\rho_0=\rho$ and,
  if possible, $[\rho_{n+1}]\cap K(t*n)=\emptyset$; if such a~$\rho_{n+1}$
  cannot be found then $K(t*n)\cap[\rho_n]$ is $G_\delta$-dense in~$[\rho_n]$
  and we are done.
  So assume that the recursion does not stop and set $\varrho=\bigcup_n\rho_n$;
  then $[\varrho]$ is disjoint from $\bigcup_nK(t*n)$, which is equal
  to~$K(t)$.
  This would contradict $G_\delta$-density of~$K(t)$ in~$[\rho]$.

  We can define $h\in\PP$ recursively by $h(n)=m_{h\restr n}$, together
  with an increasing sequence $\omegaseq\rho$ in~$\ourtree$
  such that $K(h\restr n)\cap[\rho_n]$ is $G_\delta$-dense in~$[\rho_n]$.
  Let $\rho=\bigcup_n\rho_n$, then $K(h\restr n)\cap[\rho]$
  is $G_\delta$-dense in~$[\rho]$ for all~$n$.

  Let $\delta=\dom\rho$ and let $s\in\Fn(\deltaomegaone,2)$.
  We know that $K(h\restr n)\cap[\rho]\cap[s]\neq\emptyset$ for all~$n$.
  So for every~$n$ we can take $h_{s,n}\in\PP$ that extends~$h\restr n$
  and is such that $K_{h_{s,n}}\cap[\rho]\cap[s]\neq\emptyset$.
  Because $h_{s,n}(m)=h(m)$ if $n>m$ we can define $h_s\in\PP$
  by $h_s(m)=\max_nh_{s,n}(m)$.

  As $\bee>\aleph_1$ we can find $f\ge h$ such that $h_s\le^* f$ for all~$s$.
  We claim that $K_f\cap[\rho]\cap[s]\neq\emptyset$ for all~$s$, so
  that $[\rho]\subseteq K_f$
  (the closed set $K_f\cap[\rho]$ is dense in~$[\rho]$).
  
  To see this take an~$s$ and let $n$ be such that $f(m)\ge h_s(m)$
  for $m\ge n$.
  It follows that $f(m)\ge h(m)=h_{s,n}(m)$ for $m\le n$ and
  $f(m)\ge h(m)\ge h_{s,n}(m)$ for $m\ge n$.
  This implies that $K_f$ meets $[\rho]\cap[s]$.
\end{proof}

\begin{remark}\label{rem.BIG}
  The previous result is valid for all BIG sets:
  simply work inside $[\rho]$, where $\rho$~is as in the conclusion of
  Lemma~\ref{lemma.BIGisbig}.
\end{remark}

\begin{remark}\label{rem.BIGpreim}
  Lemma~\ref{lemma.KfBIG} generalises itself to the following situation:
  let $X$ be compact, let $\varphi:X\to\ourprod$ be continuous and
  onto, and let $\langle K_f:f\in\PP\rangle$ be an order-preserving cover
  of~$X$ by compact sets.
  Then there is an~$f$ such that $\varphi[K_f]$ is~BIG.

  One can go one step further: take a closed subset~$Y$ of~$X$ such that
  $\varphi[Y]$~is BIG and conclude that for some~$f\in\PP$
  the image $\varphi[Y\cap K_f]$ is BIG.
  Simply take $\rho$ such that $[\rho]\subseteq \varphi[Y]$ and work
  in the compact space $Y\cap\varphi\preim\bigl[[\rho]\bigr]$.
\end{remark}

\begin{remark}
  The reader may have pondered the need to consider three cases in the
  proof of Lemma~\ref{lemma.KfBIG}.
  The cases $\dee=\aleph_1$ and $\bee>\aleph_1$ lead to fairly
  straightforward arguments because each give one a definite handle on
  things, be it a cofinal set of size~$\aleph_1$ or the knowledge that
  \emph{all} $\aleph_1$-sized sets are bounded.
  The intermediate case, with just one unbounded set of size~$\aleph_1$,
  is saved by Todor\v{c}evi\'c's non-trivial translation of such a set into
  a subset of~$\ourprod$ that is already quite big.

  It would be interesting to see if Lemma~\ref{lemma.KfBIG}
  can be proved using just one argument.
\end{remark}

\section*{The main result}

Now we show that that a compact space with a $\PP$-diagonal
does not admit a continuous map onto~$[0,1]^{\omega_1}$
and deduce our main result.

\begin{theorem}\label{thm.Pdiag.onto}%[$\bee=\aleph_1$]
  Assume $X$ is a compact space that maps onto~$\ourprod$.
  Then $X$ does not have a $\PP$-diagonal.
  \end{theorem}

\begin{proof}
  Let $\varphi:X\to\ourprod$ be continuous and onto.
  We use Remark~\ref{rem.BIGpreim} and say that a closed subset, $Y$, of~$X$
  is BIG if its image~$\varphi[Y]$ is.
  That is, $Y$~is BIG if there is a $\delta\in\omega_1$ such that
  $Y\cap\varphi\preim\bigl[[s]\bigr]\neq\emptyset$
  for all $s\in\Fn(\deltaomegaone,2)$.

  We observe the following: if $Y$ is BIG, as witnessed by~$\delta$, then
  for every $s\in\Fn(\deltaomegaone,2)$ the intersection
  $Y\cap\varphi\preim\bigl[[s]\bigr]$ is BIG as well; this will be witnessed
  by any~$\gamma$ that contains the domain of~$s$.

  In order to prove our theorem we assume that $X$ does have a $\PP$-diagonal,
  witnessed by $\langle K_f:F\in\PP\rangle$, and reach a contradiction.

  In order for the final recursion in the proof to succeed we need some
  preparation.
  Enumerate $\omega^{<\omega}$ in a one-to-one fashion as
  $\omegaseq{t}$, say in such a way that
  $t_m\subseteq t_n$ implies $m\le n$ (so that $t_0=\emptyseq$).
  We set $Z_0=X$ and given a BIG set~$Z_n$ we determine a BIG set~$Z_{n+1}$
  as follows.
    We check if there is a BIG subset $Z$ of~$Z_n$ with the property that
  for \emph{no} point~$z$ in~$Z$ are there a BIG subset~$Y$ of~$Z$
  and an~$f\in\PP$ with $t_n\subset f$
  such that $\{z\}\times Y\subseteq K_f$.
  If there is such a~$Z$ then every BIG subset of it also has this property
  so we can pick one that is a proper subset of~$Z_n$ and let it be $Z_{n+1}$;
  if there is no such~$Z$ then $Z_{n+1}=Z_n$.
  In the end we set $Y=\bigcap_nZ_n$.
  The set~$Y$ is BIG: for each~$n$ we have $\gamma_n\in\omega_1$ witnessing
  BIGness of~$Z_n$, then $\delta_0=\sup_n\gamma_n$ will witness BIGness of~$Y$.

  Pick $y_0\in Y$,
  take $i_0\in2$ distinct from $\varphi(y_0)(\delta_0)$,
  let $s_0=\{\orpr{\delta_0}{i_0}\}$, and
  set $Y_0=Y\cap\varphi\preim\bigl[[s_0]\bigr]$.
  By the observation above, $Y_0$~is BIG.
  Also: $\varphi(y_0)\notin\varphi[Y_0]$, so that
  $\{y_0\}\times Y_0$ is disjoint from the diagonal, $\Delta$,
  of~$X$.
  By Remark~\ref{rem.BIGpreim} we can find a BIG subset~$Y_1$ of~$Y_0$
  and $f_0\in\PP$ such
  that $\{y_0\}\times Y_1\subseteq K_{f_0}$.

  The point $y_0$ belongs to all~$Z_n$ and for any~$n$ such that
  $t_n\ge f_0$ (meaning that $t_n(i)\ge f_0(i)$ for $i\in\dom t_n$)
  it, the point~$y_0$, witnesses that $Z_{n+1}=Z_n$ in the following sense.
  The reason for having $Z_{n+1}$ be a proper subset of~$Z_n$ would be
  that for all $z\in Z$ and all BIG $Z'\subseteq Z$ and all $f\in\PP$
  with $t_n\subseteq f$ we would have $\{z\}\times Z'\not\subseteq K_f$.
  However, $y_0$ and $Y_1$ and $f_0$ show that this did not happen.

  The conclusion therefore is that for every such $t_n$ we know that
  every BIG $Z\subseteq Y$ does have an element~$z$ and a BIG subset~$Z'$
  such that $\{z\}\times Z'\subseteq K_f$ for some $f\in\PP$ that extends~$t_n$.

  This allows us to construct sequences $\omegaseq{y}$ (points in~$Y$),
  $\omegaseq{Y}$ (BIG subsets of~$Y$), and $\omegaseq{f}$ (in~$\PP$)
  such that
  \begin{enumerate}
  \item $y_n\in Y_n$, except for $n=0$,
  \item $Y_{n+1}\subseteq Y_n$,
  \item $\{y_n\}\times Y_{n+1}\subseteq K_{f_n}$,
  \item $f_{n+1}\ge f_n$ and $f_{n+1}\supseteq f_n\restr(n+1)$    
  \end{enumerate}
  As before we note that $f_m(n)=f_n(n)$ whenever $m\ge n$, so we can define
  a function~$f\in\PP$ by $f(n)=\max\{f_m(n):m\in\omega\}$.
  Note that $f\ge f_n$ for all~$n$ so that
  $$
  \{y_n\}\times Y_{n+1}\subseteq K_{f_n}\subseteq K_f
  $$
  for all~$n$.

  It follows that $\orpr{y_m}{y_n}\in K_f$ whenever $m<n$.
  This shows that $\orpr{y_m}y\in K_f$ whenever $m\in\omega$ and
  $y$~is a cluster point of $\omegaseq{y}$.
  But then $\orpr yy\in K_f$ for every cluster point~$y$ of $\omegaseq{y}$.
  However, $K_f$~was assumed to be disjoint from the diagonal of~$X$.
\end{proof}

We collect all previous results in the proof of our main theorem.

\begin{theorem}
Every compact space with a $\PP$-diagonal is metrizable.  
\end{theorem}

\begin{proof}
  As noted in the introduction the authors of~\cite{MR3338973}
  proved that a non-metrizable compact space with a $\PP$-diagonal
  will map onto the Tychonoff cube~$[0,1]^{\omega_1}$ or, equivalently,
  that it has a closed subset that maps onto~$\ourprod$.

  However that closed subset would be a compact space with a $\PP$-diagonal
  that \emph{does} map onto~$\ourprod$.
  Theorem~\ref{thm.Pdiag.onto} says that this is impossible.
\end{proof}

\end{document}